\documentclass[11pt,a4paper]{article}
\usepackage[latin1]{inputenc}
\usepackage{amsmath}
\usepackage{amsfonts}
\usepackage{amssymb}
\usepackage{amsthm}
\usepackage{verbatim}
\usepackage{mathtools}
\usepackage{fancyhdr}
\usepackage{enumerate}

\usepackage[demo]{graphicx}
\usepackage{caption}
\usepackage{subcaption}

\usepackage[T1]{fontenc}

\usepackage{graphicx, color}
\usepackage{tikz}
\usetikzlibrary{shapes.geometric}
\usetikzlibrary{calc}
\usetikzlibrary{decorations.pathreplacing}

\usepackage[numbers,sort&compress]{natbib}

\usepackage{tikz}
\usetikzlibrary{calc}

\usepackage{float}

\newtheorem{theorem}{Theorem}
\newtheorem{lemma}[theorem]{Lemma}

\newtheorem{cor}{Corollary}
\newtheorem*{claim*}{Claim}

\theoremstyle{definition}

\DeclarePairedDelimiter\floorsmall{\lfloor}{\rfloor}
\DeclarePairedDelimiter\ceilsmall{\lceil}{\rceil}

\newcommand{\floor}[1]{\floorsmall*{#1}}
\newcommand{\ceil}[1]{\ceilsmall*{#1}}

\newcommand{\lr}[1]{\left(#1\right)}

\renewcommand{\leq}{\leqslant}
\renewcommand{\geq}{\geqslant}

\newcommand{\x}{*1.8}
\newcommand{\y}{*2.4}

\usepackage{todonotes}

\usepackage{graphicx}
\usepackage[left=3cm, right=3cm, top=2.50cm, bottom=2.50cm]{geometry}

\title{\textbf{Anagram-free colourings of graph subdivisions}}
\author{Tim E. Wilson\quad David R. Wood\footnote{Research supported by Australian Research Council}\\
\\
School of Mathematical Sciences\\
Monash University\\
Melbourne, Australia\\
\texttt{\{timothy.e.wilson, david.wood\}@monash.edu}}

\begin{document}

	\maketitle
	
	An \emph{anagram} is a word of the form $WP$ where $W$ is a non-empty word and $P$ is a permutation of $W$. A vertex colouring of a graph is \textit{anagram-free} if no subpath of the graph is an anagram. Anagram-free graph colouring was independently introduced by Kam{\v{c}}ev, {\L}uczak and Sudakov 
	and ourselves
	. In this paper we introduce the study of anagram-free colourings of graph subdivisions. We show that every graph has an anagram-free $8$-colourable subdivision. The number of division vertices per edge is exponential in the number of edges. For trees, we construct anagram-free $10$-colourable subdivisions with fewer division vertices per edge. Conversely, we prove lower bounds, in terms of division vertices per edge, on the anagram-free chromatic number for subdivisions of the complete graph and subdivisions of complete trees of bounded degree.
	 
\section{Introduction}

An \emph{anagram} is a word of the form $WP$ where $W$ is a non-empty word and $P$ is a permutation of $W$. A vertex colouring of a graph is \emph{anagram-free} if the sequence of colours on every path in the graph is not an anagram. The \textit{anagram-free chromatic number}, $\phi(G)$, of a graph $G$, is the minimum number of colours in an anagram-free colouring of $G$. \citet*{alon2002nonrepetitive} proposed anagram-free colouring as a subject of study as a generalization of square-free colouring. A \textit{square} is a word of the form $WW$ where $W$ is a non-empty word. A graph colouring is \textit{square-free} if the sequence of colours on every path in the graph is not a square. A square-free graph colouring is also called a \textit{nonrepetitive colouring}. The \textit{square-free chromatic number}, $\pi(G)$, of a graph $G$, is the minimum number of colours in a square-free colouring of $G$. 

Square-free words and anagram-free words both originate from the study of the combinatorics of words. Square-free words are known as \textit{nonrepetitive} words and anagram-free words are known as \textit{abelian square-free} or \textit{strongly nonrepetitive}. Both types of word can be arbitrarily long with a bounded number of distinct symbols. In particular, \citet{thue1914probleme} constructed arbitrarily long square-free words on $3$ symbols. \citet{keranen1992abelian, keranen2009abelian} constructed arbitrarily long anagram-free words on $4$ symbols. The longest square-free or anagram-free words on two symbols have length $3$. The longest anagram-free words on three symbols have length $7$ \citep{cummings1996strongly}. Words are equivalent to colourings of paths, so $\pi(P) \leq 3$ and $\phi(P) \leq 4$ for all paths $P$.

Square-free colouring was introduced by \citet{alon2002nonrepetitive} and has since received much attention \cite{britten1971repetitive,grytczuk2007nonrepetitive, grytczuk2013new, grytczuk2006nonrepetitive, harant2012nonrepetitive, dujmovic2011nonrepetitive, brevsar2007nonrepetitivetree}. A central area of study has been to bound $\pi(G)$ by a function of maximum degree, $\Delta(G)$. \citet{alon2002nonrepetitive} proved a result that implies $\pi(G) \leq c\Delta(G)^2$ for some constant $c$. Several subsequent works improved the value of $c$ \cite{grytczuk2006nonrepetitive, harant2012nonrepetitive} with the best known value being $c = 1 + o(1)$ \cite{dujmovic2011nonrepetitive}. Lower bounds for square-free colouring apply to anagram-free colouring because $\phi(G) \geq \pi(G)$ for all graphs $G$. Indeed, a square is an anagram with the identity permutation, so for a colouring to be anagram-free it must also be square-free. Anagram-free colourings were recently introduced by \citet*{kamvcev2016anagram} and \citet{wilson2016abelian} both proving, among other results, that $\phi$ is not bounded by a function of maximum degree.

In this paper we study $\phi$ on graph subdivisions, with a focus on constructing subdivisions with bounded anagram-free chromatic number. A \textit{subdivision} of a graph, $G$, is a graph obtained from $G$ by replacing each edge $vw \in E(G)$ by a path with endpoints $vw$. If an edge $uv$ of $G$ is replaced by a path $uw_1w_2\ldots w_{i - 1}v$ of length $i$, then we say that $uv$ was \textit{subdivided $i$ times} and call the vertices $w_1,\ldots,w_{i-1}$ \textit{division vertices}. The \textit{$k$-subdivision} of $G$ is the subdivision in which every edge of $G$ is subdivided exactly $k$ times. Similarly, a \textit{$(\leq k)$-subdivision} of $G$ is a subdivision in which every edge of $G$ is subdivided at most $k$ times. Graphs with many division vertices are locally paths or stars, so one would expect highly subdivided graphs to  have relatively low have anagram-free chromatic number. Square-free colouring has been studied on subdivisions of graphs and here this intuition is known to hold. \citet{grytczuk2007nonrepetitive} showed that every graph has a subdivision, $S$, with $\pi(S) \leq 5$ with the bound later improved to $4$ by \citet{barat2008notes}, and finally to $3$ by \citet{pezarski2009non}.

Before introducing our results, we summarise the known results for $\pi$ and $\phi$ on trees. For a rooted tree, $T$, with root $r$, the \textit{depth} of a vertex $v$ in $T$ is the distance between $v$ and $r$. A \textit{$d$-ary tree} is a rooted tree with at most $d$ children per vertex. The \textit{complete $d$-ary tree of height $h$} is the rooted tree such that every non-leaf vertex has $d$ children and every leaf has depth $h$. The complete $2$-ary tree is called the \textit{complete binary tree}. \citet{brevsar2007nonrepetitivetree} studied square-free colourings of trees, showing that $\pi(T) \leq 4$ for every tree $T$, and that $T$ has a subdivision, $S$, with $\pi(S) \leq 3$. By contrast, $\phi$ is unbounded on trees \cite{wilson2016abelian, kamvcev2016anagram}. In particular, \citet{kamvcev2016anagram} prove the following bounds for the complete binary tree.
\begin{theorem}[\citet{kamvcev2016anagram}]\label{thm:completeBinaryTree}
	Let $T_h$ be the complete binary tree of height $h$. Then
	\begin{align*}
		\sqrt{\frac{h}{\log_2 h}} \leq \phi(T_h) \leq h + 1.
	\end{align*}
\end{theorem}
The upper bound, $\phi(T) \leq h + 1$, holds for every tree, $T$, of height $h$, and is obtained by colouring vertices by their depth. \citet{wilson2016abelian} show that this upper bound is almost best possible on general trees by proving that $\phi(T) \geq h$ where $T$ is the $(h - 1)^h$-ary tree of height $h$.

\subsection{Subdivisions of Trees}\label{sec:tree}

We now introduce the results in the present paper. Our results complement the bounds on $\phi$ for trees proved in \cite{wilson2016abelian, kamvcev2016anagram}. We construct anagram-free $8$-colourable subdivisions of binary trees.
\begin{theorem}\label{thm:sub_binaryTree}
	Every binary tree, $T$, of height $h$, has a $(\leq 3^{h - 1} - 1)$-subdivision, $S$, with $\phi(S) \leq 8$.
\end{theorem}
More generally, we construct anagram-free $10$-colourable subdivisions of $d$-ary trees.
\begin{theorem}\label{thm:colorDaryTree_simple}
	Every $d$-ary tree, $T$, of height $h$, has a $\lr{\leq 2d(d+1)^{h-1}}$-subdivision, $S$, with $\phi(S) \leq 10$.
\end{theorem}
The number of division vertices per edge is exponential in the height for both Theorem~\ref{thm:sub_binaryTree} and Theorem \ref{thm:colorDaryTree_simple}. This raises the question of whether better constructions exist. In particular, does every tree of bounded degree have an anagram-free $c$-colourable subdivision with the number of division vertices per edge growing slower than exponentially with height? We answer this question in the negative with the lower bound in the following theorem.
\begin{theorem}\label{thm:dary_sub_chrom}
	The $k$-subdivision, $S$, of the complete $d$-ary tree of height $h$ satisfies
	\begin{align*}
		\sqrt{\frac{h}{\log_{\min\{d, (h(k + 1))^2\}}(h(k + 1))}} \leq \phi(S) \leq \frac{10h}{\log_{d + 1}\lr{k/2d}} + 14.
	\end{align*}
\end{theorem}
Theorem \ref{thm:dary_sub_chrom} implies that, for sufficiently large height $h$, the number of division vertices per edge in an anagram-free $c$-colourable subdivision of the complete $d$-ary tree is at least
\begin{align*}
	k \geq \frac{d^{h/c^2}}{h} - 1,
\end{align*}
which is exponential in $h$ for fixed $c$. The upper bound in Theorem \ref{thm:dary_sub_chrom} is obtained by applying Theorem \ref{thm:colorDaryTree_simple} to appropriate subtrees of the complete $d$-ary tree. The lower bound is a generalization of Theorem \ref{thm:completeBinaryTree}; see Theorem \ref{thm:effectiveDepthDegree} for details.

\subsection{Subdivisions of General Graphs}

We also study $\phi$ on subdivisions of general graphs, and prove the following theorems in this direction. The first has fewer division vertices per edge, while the second has fewer colours.
\begin{theorem}\label{thm:graphSub_constant_nice}
	Every graph $G$ has a $(\leq 3(2)^{2|E(G)| - 1} - 1)$-subdivision, $S$, with $\phi(S) \leq 14$.
\end{theorem}
\begin{theorem}\label{thm:graphSub_constant_optimized}
	Every graph $G$ has a $\lr{\leq 45\lr{\frac{75}{9} + 1}^{2|E(G)| - 1}}$-subdivision, $S$, with $\phi(S) \leq 8$.
\end{theorem}
The bound $\phi(S) \leq 8$ in Theorem \ref{thm:graphSub_constant_optimized} is our best bound on $\phi$, notably better than the bound for subdivisions of trees (Theorem \ref{thm:colorDaryTree_simple}). On the other hand, Theorem \ref{thm:colorDaryTree_simple} uses fewer division vertices. Indeed, if $T$ is the complete $d$-ary tree, then the number of division vertices per edge is polynomial in $|E(T)|$.

To investigate the optimality, in terms of division vertices per edge, of Theorem \ref{thm:graphSub_constant_nice} and Theorem \ref{thm:graphSub_constant_optimized}, we prove a lower bound on $\phi(K_n)$, the complete graph on $n$ vertices. Such results exist for $\pi$, in particular, \citet{nevsetvril2012characterisations} proved the following theorem.
\begin{theorem}[\citet{nevsetvril2012characterisations}]\label{thm:pi_K_n}
	For $k \geq 2$, the $k$-subdivision of $K_n$, denoted $S$, satisfies
	\begin{align*}
	\lr{\frac{n}{2}}^{1/(k + 1)} \leq \pi(S) \leq 9 \ceil{n^{1/(k + 1)}}.
	\end{align*}
\end{theorem}
Since $\phi(G) \geq \pi(G)$, the lower bound in Theorem \ref{thm:pi_K_n} implies that $k \geq \log_c \lr{n/2} - 1$ for every anagram-free $c$-colourable $k$-subdivision of $K_n$. We prove the following improvement.
\begin{theorem}\label{thm:K_n_anagram_lowerBound}
	Let $S$ be a $\lr{\leq k}$-subdivision of $K_n$. If $S$ is anagram-free $c$-colourable then
	\begin{align*}
		k \geq \lr{c!\lr{\frac{n}{c} - 1}}^{1/c} - c.
	\end{align*}
\end{theorem}
For fixed $c$, the bound in Theorem \ref{thm:K_n_anagram_lowerBound} is $k \geq \Omega\lr{n^{1/c}}$, which is larger than the logarithmic bound implied by Theorem \ref{thm:pi_K_n}. Still, this lower bound is much less than the exponential upper bound implied by Theorem \ref{thm:graphSub_constant_nice}. We expect that both our upper and lower bounds on $k$ can be significantly improved.

\section{Basic Observations}

This section contains basic observations and definitions that will be used throughout the rest of the paper. A \textit{colour multiset} of size $n$ on $c$ colours is a multiset of size $n$ with entries from $[c] := \{1,2,\ldots,c\}$. Let $\mathcal{M}_{n,c}$ be the set of all colour multisets of size $n$ on $c$ colours, and let $\mathcal{M}_{\leq n,c}$ be the set of all colour multisets of size at most $n$ on $c$ colours. For a coloured graph $G$ define the following. Let $M(G)$ be the multiset of colours assigned to the vertices of $G$. For a subset, $C$, of the colours, let $M_C(G)$ be $M(G)$ restricted to $C$. Let $V_C(G)$ be the vertices of $G$ that have a colour from $C$. 

Call a path \textit{even} if it has an even number of vertices. Define $LR$ to be the \textit{split} of an even path, $P$, if $|L| = |R|$ and $P = LR$. Note that a coloured path, $P$, is an anagram if and only if $M(L) = M(R)$. Equivalently, $P$ is not an anagram if $M_C(L) \neq M_C(R)$ for some set of colours $C$. For a path $P$ and set of colours $C$, define \textit{$P$ restricted to $C$} to be the word $\omega_C(P) := f(v_1)f(v_2)\ldots f(v_x)$, where $v_1,v_2,\ldots,v_x$ are the vertices in $V_C(P)$, in the order defined by $P$, and $f$ is the vertex colouring of $P$. Similarly, for a word $W = w_1w_2\dots w_n$ and set of symbols $C$, define \textit{$W$ restricted to $C$} to be $\omega_C(W) := f(w_1)f(w_2)\dots f(w_n)$, where $f(w) = w$ if $w \in C$ and $f(w)$ is the empty character otherwise. We use these observations in the form of the following lemma.
\begin{lemma}\label{lem:anagramFreeSubColor}
	A path, $P$, coloured by $C$, is an anagram if and only if for all $C' \subseteq C$, $P$ restricted to $C'$ is an anagram or the empty word.
\end{lemma}
\begin{proof}
	We first prove the forward implication. Let $C' \subseteq C$ be such that $\omega_{C'}(P)$ is nonempty, since the empty case is trivial. Let $LR$ be the split of $P$. Note that $M_{C'}(L) = M(\omega_{C'}(L))$ and $M_{C'}(R) = M(\omega_{C'}(R))$. Since $P$ is an anagram,
	\begin{align*}
		M(\omega_{C'}(L)) = M_{C'}(L) = M_{C'}(R) = M(\omega_{C'}(R)).
	\end{align*}
	Therefore $\omega_{C'}(P)$ is an anagram.

	To prove the back implication take $C' = C$. Then $P$ restricted to $C'$, which is all of $P$, is an anagram.
\end{proof}
The \textit{midedge} of an even path $P$ with split $LR$ is the edge of $P$ not contained in $L$ or $R$. For a connected graph $G$, define the \textit{distance} between an edge $ab$ and a vertex $v$ to be the minimum of $\text{dist}(a,x)$ and $\text{dist}(b,x)$. 

\section{Subdivisions of trees}

This section contains our results for trees. For every vertex $v$ in a rooted tree $T$, define $A_T(v)$ to be the set of ancestors and descendants of $v$ in $T$. A \textit{branch vertex} is a vertex of a rooted tree with at least two children.

\subsection{Subdivisions of binary trees}

\begingroup
\def\thetheorem{\ref{thm:sub_binaryTree}}
\begin{theorem}
	Every binary tree, $T$, of height $h$, has a $(\leq 3^{h - 1} - 1)$-subdivision, $S$, with $\phi(S) \leq 8$.
\end{theorem}
\addtocounter{theorem}{-1}
\endgroup

\begin{proof}
	$2$-colour the edges of $T$ with $\{1, 2\}$ such that for every branch vertex, $v \in V(T)$, with children $c_1$ and $c_2$, the edge $vc_i$ receives colour $i$. Colour the remaining edges arbitrarily from $\{1, 2\}$. Let $S$ be the subdivision of $T$ such that edges at distance $x$ from the root are subdivided $3^{h - x - 1} - 1$ times. Note that edges incident with leaves of depth $h$ are not subdivided.
	
	Let $r$ be the root of $S$. Label the vertices of $S$ according to the edge $2$-colouring of $T$ as follows:
	\begin{itemize}
		\item Label division vertices with the colour of the corresponding edge in $T$.
		\item Label $r$ with $1$.
		\item Label the original non-root vertices with the label of their parent edge in $T$.
	\end{itemize}
	
	Let $W = w_1w_2\ldots$ be an anagram-free word on $\{1,2,3,4\}$.  Define $V_\ell(S)$ to be the set of vertices with label $\ell$ in $S$. Colour every vertex $v \in V(S)$ by $(\ell,w_x)$ where $\ell$ is the label of $v$ and $x$ is the number of vertices with label $\ell$ on the $vr$-path. We now show that this $8$-colouring of $S$ is anagram-free.
	
	Let $P$ be an even order path in $S$. Consider the case where there is some $\ell \in \{1,2\}$ such that $A_S(v) = V_\ell(P)$ for all vertices $v \in V_\ell(P)$. If $V_\ell(P) = \emptyset$ then $P$ is not an anagram because, by construction, $S$ restricted to a label is anagram-free. So now consider $V_\ell(P) \neq \emptyset$ and let $C' = \{\ell\}\times\{1,2,3,4\}$. Then $\omega_{C'}(P) = (\ell, w_y)(\ell, w_{y + 1})\ldots(\ell, w_{y + |V_\ell(P)|})$, for some integer $y$, because the number of $\ell$ labelled vertices on the $vr$-path increments by $1$ for all vertices $v \in V(P_\ell)$ along $P$. Therefore $\omega_{C'}(P)$ is a subword of $W$ so, by Lemma \ref{lem:anagramFreeSubColor}, $P$ is not an anagram.

	Now consider the case where for every $\ell \in \{1,2\}$ there exists a $v \in V_\ell(P)$ such that $A_S(v) \neq V_\ell(P)$. Let $u$ be the minimum depth vertex in $V(P)$. Both labels have vertices that are not mutual ancestors or descendants so $u$ has two children in $T$, $x,y \in V(T)$, and, in addition, $x, y \in V(P)$.
	
	Partition $V(P)$ into $X := \lr{V(P) \cap A_S(x)} \setminus \{u\}$ and $Y := V(P) \cap A_S(y)$. Let $LR$ be the split of $P$ such that, without loss of generality, $u \in V(R)$. Also without loss of generality, choose $x$ and $y$ such that $Y \subseteq V(R)$. Note that $Y \cap V(L) = \emptyset$ so
	\begin{align*}
	|X \cap V(L)| = |V(L)| = |V(R)| = |X \cap V(R)|  + |Y \cap V(R)|.
	\end{align*}
	
	Let $z$ be the integer such that $3^z + 1$ is the order of the $ux$-path in $S$. We will prove an upper bound on $|X \cap V(R)|$ to show that the midedge of $P$ is `close' to $u$. First, note that 
	\begin{align*}
	|Y \cap V(R)| \geq 3^z + 2
	\end{align*}
	because the edge $uy$ was subdivided $3^z - 1$ times. $|X|$ is at most the length of a path from $u$ to a leaf so
	\begin{align*}
	|X \cap V(L)| &\leq 3^z + 3^{z-1} + \dots + 3^1 + 1 - |X \cap V(R)| = \frac{3}{2}3^z -\frac{1}{2} - |X \cap V(R)|.
	\end{align*}
	Therefore
	\begin{align*}
	|X \cap V(R)| &= |X \cap V(L)| - |Y \cap V(R)| \leq \frac{3}{2}3^z -\frac{1}{2} - |X \cap V(R)| -3^z - 2.
	\end{align*}
	Thus
	\begin{align*}
	|X \cap V(R)| &\leq \frac{1}{4}3^z -\frac{3}{4}.
	\end{align*}
	Without loss of generality let $x$ have label $1$ and $y$ have label $2$. Indeed, the labels of $x$ and $y$ are distinct because edges $ux$ and $uy$ have different colours in $T$. All vertices on the $ux$-path (except possibly $u$) have label $1$ so
	\begin{align*}
	|V_1(L)| &\geq 3^z - |X \cap V(R)| \geq  \frac{3}{4}3^z  + \frac{3}{4}.
	\end{align*}
	To put an upper bound on $|V_1(R)|$ assume the worse case, that $u$ has label $1$. Then
	\begin{align*}
	|V_1(R)| &\leq |X \cap V_1(R)| + 1 +  3^{z-1} + 3^{z-2} + \dots + 3^1 + 1\\
	&\leq \frac{1}{4}3^z -\frac{3}{4} + 1 + \frac{3}{2}3^{z-1} -\frac{1}{2} \\
	&= \frac{3}{4}3^z - \frac{5}{4}.
	\end{align*}
	It follows that $|V_1(R)| < |V_1(L)|$. Therefore $P$ is not an anagram.
\end{proof}

\subsection{Subdivisions of $d$-ary trees}

The construction in Theorem \ref{thm:sub_binaryTree} does not extend to a good bound on $\phi$ for subdivisions of complete $d$-ary trees. The obvious extension, using $d$ labels for the edge colouring, shows that the complete $d$-ary tree has a $4d$-colourable subdivision. We prove the following result for complete $d$-ary trees.
\begin{theorem}\label{thm:colorDaryTree}
	The complete $d$-ary tree, $T$, of height $h$, has a  $\lr{\leq 2d(d+1)^{h-1}}$-subdivision, $S$, with $\phi(S) \leq 10$.
\end{theorem}
\begin{proof}
	Let $r$ be the root of $T$. For all $x \in [h]$ and $y \in [d]$ let $t_{x,y} = y(d + 1)^{x - 1}$. Define the labelling $\ell : E(T) \to [d]$ such that edges incident with the same parent vertex receive distinct labels. Let $S$ be the subdivision of $T$ such that every edge $e \in E(T)$ is subdivided $2t_{h - z, \ell(e)}$ times where $z$ is the depth of $e$. Note that $z \in \{0, \ldots, h - 1\}$.
	
	Let $\ell_T : V(T) \to \{\text{black}, \text{white}\}$ be a proper vertex $2$-colouring of $T$. Define the labelling $\ell_S : V(S) \to \{\text{black}, \text{white}, \text{red}, \text{green}\}$ as follows. If $v \in V(S)$ is an original vertex then $\ell_S(v) := \ell_T(v)$. Otherwise, let $v' \in V(S)$ be the closest original vertex to $v$ and $e \in E(T)$ be the edge such that $v$ is a division vertex of $e$. If $v'$ is the parent of $e$ then $\ell_S(v) := \text{red}$, otherwise $\ell_S(v) := \text{green}$. Note that $v'$ is well defined because all edges of $T$ have an even number of division vertices. See Figure \ref{fig:daryTreeSubdiv} for an example of this construction.
	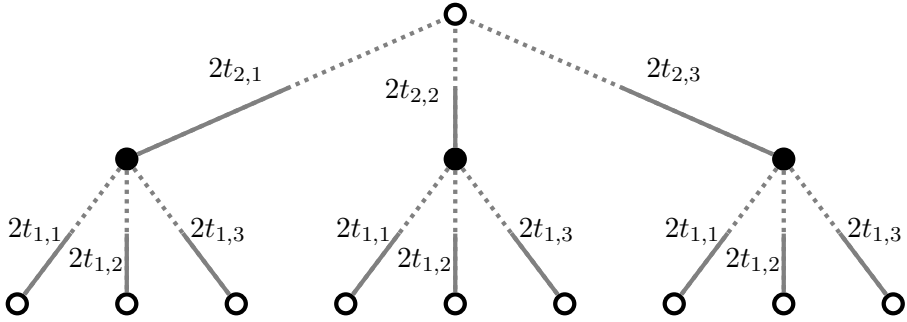
\begin{figure}[h]
		\begin{center}
			\begin{tikzpicture}[
				line width=1.6pt,
				vertex/.style={circle,inner sep=0pt,minimum size=0.25cm}, 
				point/.style={circle,inner sep=0pt,minimum size=0cm}, 
				scale = 0.4
				]
			
				\node[draw = black, fill = white] (b) at ($(0\x,-2\y)$) [vertex] {};
			
				\node[draw = white, fill = white] (o1) at ($(-3\x,-3\y)$) [point] {};
				\node[draw = white, fill = white] (o2) at ($(0\x,-3\y)$) [point] {};
				\node[draw = white, fill = white] (o3) at ($(3\x,-3\y)$) [point] {};
				
				\node[draw = white, fill = white] () at ($(-3\x - 1.8,-3\y + 0.5)$) [point] {$2t_{2,1}$};
				\node[draw = white, fill = white] () at ($(0\x - 1.4,-3\y - 0.2)$) [point] {$2t_{2,2}$};
				\node[draw = white, fill = white] () at ($(3\x + 1.8,-3\y + 0.5)$) [point] {$2t_{2,3}$};
				
				\node[draw = black, fill = black] (b1) at ($(-6\x,-4\y)$) [vertex] {};
				\node[draw = black, fill = black] (b2) at ($(0\x,-4\y)$) [vertex] {};
				\node[draw = black, fill = black] (b3) at ($(6\x,-4\y)$) [vertex] {};
				
				\node[draw = black, fill = white] (o11) at ($(-7\x,-5\y)$) [point] {};
				\node[draw = black, fill = white] (o12) at ($(-6\x,-5\y)$) [point] {};
				\node[draw = black, fill = white] (o13) at ($(-5\x,-5\y)$) [point] {};
				\node[draw = black, fill = white] (o21) at ($(-1\x,-5\y)$) [point] {};
				\node[draw = black, fill = white] (o22) at ($(0\x,-5\y)$) [point] {};
				\node[draw = black, fill = white] (o23) at ($(1\x,-5\y)$) [point] {};
				\node[draw = black, fill = white] (o31) at ($(5\x,-5\y)$) [point] {};
				\node[draw = black, fill = white] (o32) at ($(6\x,-5\y)$) [point] {};
				\node[draw = black, fill = white] (o33) at ($(7\x,-5\y)$) [point] {};
				
				\node[draw = white, fill = white] () at ($(-7\x - 1.2,-5\y)$) [point] {$2t_{1,1}$};
				\node[draw = white, fill = white] () at ($(-6\x - 1,-5\y - 1)$) [point] {$2t_{1,2}$};
				\node[draw = white, fill = white] () at ($(-5\x + 1.2,-5\y)$) [point] {$2t_{1,3}$};
				\node[draw = white, fill = white] () at ($(-1\x - 1.2,-5\y)$) [point] {$2t_{1,1}$};
				\node[draw = white, fill = white] () at ($(0\x - 1,-5\y - 1)$) [point] {$2t_{1,2}$};
				\node[draw = white, fill = white] () at ($(1\x + 1.2,-5\y)$) [point] {$2t_{1,3}$};
				\node[draw = white, fill = white] () at ($(5\x - 1.2,-5\y)$) [point] {$2t_{1,1}$};
				\node[draw = white, fill = white] () at ($(6\x - 1,-5\y - 1)$) [point] {$2t_{1,2}$};
				\node[draw = white, fill = white] () at ($(7\x + 1.2,-5\y)$) [point] {$2t_{1,3}$};
				
				\node[draw = black, fill = white] (b11) at ($(-8\x,-6\y)$) [vertex] {};
				\node[draw = black, fill = white] (b12) at ($(-6\x,-6\y)$) [vertex] {};
				\node[draw = black, fill = white] (b13) at ($(-4\x,-6\y)$) [vertex] {};
				\node[draw = black, fill = white] (b21) at ($(-2\x,-6\y)$) [vertex] {};
				\node[draw = black, fill = white] (b22) at ($(0\x,-6\y)$) [vertex] {};
				\node[draw = black, fill = white] (b23) at ($(2\x,-6\y)$) [vertex] {};
				\node[draw = black, fill = white] (b31) at ($(4\x,-6\y)$) [vertex] {};
				\node[draw = black, fill = white] (b32) at ($(6\x,-6\y)$) [vertex] {};
				\node[draw = black, fill = white] (b33) at ($(8\x,-6\y)$) [vertex] {};
				
				\draw[draw=gray, dotted](b1)--(b);
				\draw[draw=gray, dotted](b2)--(b);
				\draw[draw=gray, dotted](b3)--(b);
				
				\draw[draw=gray](b1)--(o1);
				\draw[draw=gray](b2)--(o2);
				\draw[draw=gray](b3)--(o3);
				
				\draw[draw=gray, dotted](b1)--(b11);
				\draw[draw=gray, dotted](b1)--(b12);
				\draw[draw=gray, dotted](b1)--(b13);
				\draw[draw=gray, dotted](b2)--(b21);
				\draw[draw=gray, dotted](b2)--(b22);
				\draw[draw=gray, dotted](b2)--(b23);
				\draw[draw=gray, dotted](b3)--(b31);
				\draw[draw=gray, dotted](b3)--(b32);
				\draw[draw=gray, dotted](b3)--(b33);
				
				\draw[draw=gray](b11)--(o11);
				\draw[draw=gray](b12)--(o12);
				\draw[draw=gray](b13)--(o13);
				\draw[draw=gray](b21)--(o21);
				\draw[draw=gray](b22)--(o22);
				\draw[draw=gray](b23)--(o23);
				\draw[draw=gray](b31)--(o31);
				\draw[draw=gray](b32)--(o32);
				\draw[draw=gray](b33)--(o33);
			\end{tikzpicture}
		\end{center}
		\caption{$S$ for the complete $3$-ary tree of height $2$. The edges represent a number of division vertices denoted by their label. Each edge has a dotted half representing its red division vertices and a solid half representing its green division vertices.}
		\label{fig:daryTreeSubdiv}
	\end{figure}
	
	Define the \textit{red-depth} of a vertex $v \in V(S)$ to be the number of red vertices on the $vr$-path in $S$ and define \textit{green-depth} analogously. Let $W = w_1 w_2 \ldots$ be a long anagram-free word on $\{1,2,3,4\}$. Define the vertex colouring $f$ as follows. If $v \in V(S)$ is an original vertex then colour $v$ by $\ell_S(v)$. Otherwise, let $i$ be the $\ell_S(v)$-depth of $v$ and define $f(v) := (w_i, \ell_S(v))$. A vertex has label black or white if and only if it is an original vertex so $f$ is a $10$-colouring of $S$.
	
	Let $P$ be a path in $S$ and assume for the sake of contradiction that $P$ is an anagram. $P$ contains at least one division vertex because the original vertices have a proper colouring in $T$, and all edges not incident with leaves have at least one division vertex. Let $u$ be the vertex with minimum depth in $P$.
	
	First consider the case where $u$ is an endpoint of $P$. In this case $V(P) \subseteq A_S(v)$ for all vertices $v \in V(P)$. Without loss of generality $P$ contains a red division vertex. Therefore the red-depth increments by one for red vertices along $P$. Since the red vertices are coloured according to $W$ and their red-depth, $\omega_{\text{red}}(P)$ is a subword of $W$ and thus not an anagram. Therefore, by Lemma \ref{lem:anagramFreeSubColor}, $P$ is not an anagram.
	
	The remaining case is where $u$ is not an endpoint of $P$. In this case $u$ is an original vertex. For all $e \in E(T)$, let $D_e$ be the division vertices of $e$. Say that $P$ \textit{hits} $e$ if $D_e \cap V(P) \neq \emptyset$ and that $P$ \textit{contains} $e$ if $D_e \subseteq V(P)$. Let $\alpha$ be the largest edge in $T$ (the edge with most division vertices in $S$) hit by $P$ and $\beta$ be the second largest edge in $T$ hit by $P$. Since $t_{x,y} > t_{x', y'}$ for all $y$, $y'$, and $x > x'$ the edges of $T$ are larger nearer the root so both $\alpha$ and $\beta$ are adjacent to $u$. Let $v_\alpha$ and $v_\beta$ be the endpoints of $P$ denoted such that $uv_\alpha\text{-path}$ hits $D_\alpha$.
	
	Let $C' = \{\text{red}, \text{green}\} \times \{1,2,3,4\}$ and define $W_L$, $W_\alpha$, $W_\beta$ and $W_R$ as follows. Firstly, the concatenation $W_L W_\alpha W_\beta W_R = \omega_{C'}(P)$. $W_\alpha$ is the subword corresponding the vertices $V(P) \cap D_\alpha$. $W_L$ is the subword corresponding to the division vertices in $V(uv_\alpha\text{-path}) \setminus D_\alpha$, note that $W_L$ may be the empty word. Similarly, $W_\beta$ is the subword corresponding the vertices $V(P) \cap D_\beta$ and $W_R$ is the subword corresponding to the remaining division vertices of $P$.
	
	Let $x_\alpha$, $y_\alpha$ and $y_\beta$ be such that $|D_\alpha| = 2t_{x_\alpha, y_\alpha}$ and $|D_\beta| = 2t_{x_\alpha, y_\beta}$. Firstly, 
	\begin{align*}
		|W_L| \leq b := 2\sum_{i = 1}^{x_\alpha - 1} t_{i, d}.
	\end{align*}
	because $|W_L|$ is at most the number of division vertices on the longest path from the child of $\alpha$ to a leaf of $S$. Similarly $|W_R| \leq b$. For all $x \in [h]$,
	\begin{align*}
		t_{x,1} = 1 + \sum_{i = 1}^{x - 1} t_{i, d}
	\end{align*}
	because, by induction on $x$,
	\begin{align*}
		t_{x,1} &= (d + 1)t_{x - 1,1} = (d + 1)\lr{1 + \sum_{i = 1}^{x - 2} t_{i, d}} = (d + 1) + \sum_{i = 2}^{x - 1} t_{i, d} = 1 + \sum_{i = 1}^{x - 1} t_{i, d}.
	\end{align*}
	Therefore
	\begin{align*}
		|D_\alpha| = 2t_{x_\alpha,y_\alpha} \geq 2t_{x_\alpha,1} = 2 + 2\sum_{i = 1}^{x_\alpha - 1} t_{i, d} > b.
	\end{align*}
	Similarly $|D_\beta| > b$. Also,
	\begin{align*}
		|D_\alpha| = 2y_\alpha t_{x_\alpha,1} \geq 2y_\beta t_{x_\alpha, 1} + 2t_{x_\alpha, 1} = 2t_{x_\alpha, y_\beta} + 2 + 2\sum_{i = 1}^{x_\alpha - 1} t_{i, d} > |D_\beta| + b.
	\end{align*}
	Recall that the vertex colouring of $T$ is a proper $2$-colouring and that $V(P)$ contains an original vertex. The shortest anagram in a proper $2$-colouring has four vertices. Therefore, by Lemma \ref{lem:anagramFreeSubColor}, both $L$ and $R$ contain at least two original vertices so $P$ contains at least three edges of $T$. This implies that at least one of $W_L$ and $W_R$ are not the empty word. Thus at least one of $\alpha$ and $\beta$ is contained in $P$. Let $LR$ be the split of $P$ with $v_\alpha \in V(L)$.
	
	Consider the case where $\alpha$ is not contained in $P$. Then $|W_\beta| = |D_\beta| > b \geq |W_R|$ so $W_R$ is a subword of $\omega_{C'}(R)$. This implies that $L$ only contains one original vertex, which is a contradiction, so $P$ is not an anagram. 
	
	Now consider the case where $\alpha$ is contained in $P$. Then $|W_\alpha| = |D_\alpha|$. Since exactly half the division vertices of each edge are labelled red, 
	\begin{align*}
		|\omega_\text{red}(W_\alpha)| &= |\omega_\text{green}(W_\alpha)| = \frac{|W_\alpha|}{2}\\
		|\omega_\text{green}(W_\beta)| &\leq \frac{|W_\beta|}{2}\\
		|\omega_\text{red}(W_L)| &\leq \frac{b}{2}\\
		|\omega_\text{green}(W_R)| &\leq \frac{b}{2}.
	\end{align*}
	If all vertices corresponding to $\omega_\text{green}(W_\alpha)$ are in $L$, then
	\begin{align*}
		|\omega_\text{green}(L)| \geq |\omega_\text{green}(W_\alpha)| > \frac{|D_\beta| + b}{2} \geq |\omega_\text{green}(W_\beta)| + |\omega_\text{green}(W_R)| \geq |\omega_\text{green}(R)|.
	\end{align*}
	Thus $P$ is not an anagram. If all vertices corresponding to $\omega_\text{red}(W_\alpha)$ are in $R$, then 
	\begin{align*}
		|\omega_\text{red}(R)| \geq |\omega_\text{red}(W_\alpha)| > \frac{b}{2} \geq |\omega_\text{red}(W_L)| \geq |\omega_\text{red}(L)|.
	\end{align*}
	Thus $P$ is not an anagram. This covers all cases since $v_\alpha \in V(L)$.
\end{proof}

Theorem \ref{thm:colorDaryTree_simple} is a simple corollary of Theorem \ref{thm:colorDaryTree}.
\begingroup
\def\thetheorem{\ref{thm:colorDaryTree_simple}}
\begin{theorem}
	Every $d$-ary tree, $T$, of height $h$, has a $\lr{\leq 2d(d+1)^{h-1}}$-subdivision, $S$, with $\phi(S) \leq 10$.
\end{theorem}
\addtocounter{theorem}{-1}
\endgroup
\begin{proof}
	Apply Theorem \ref{thm:colorDaryTree} to the complete $d$-ary tree of height $h$ and take the appropriate subgraph of the resulting subdivision.
\end{proof}
The next section shows that the exponential upper bound on the number of division vertices per edge in Theorem \ref{thm:colorDaryTree_simple} is necessary.

\subsection{Lower bounds}

This subsection extends Theorem \ref{thm:completeBinaryTree}, for complete binary trees, by \citet{kamvcev2016anagram}. We generalise their method of proof to obtain a result about subdivisions of high degree trees. The following definitions are extensions of those found in their original paper.

Let $T$ be a rooted tree with root $r$.  The \emph{effective vertices} of $T$ are its leaves and branch vertices. The \emph{effective root} of $T$ is the closest effective vertex to $r$, including $r$. The \emph{effective height} of $T$ is the minimum, over the leaves of $T$, of the number of branch vertices on each root to leaf path.

Call $T$ \emph{essentially $i$-monochromatic} if all of its effective vertices are coloured $i$. Call $T$ \emph{essentially monochromatic} if it is essentially $i$-monochromatic for some $i$. For $d \geq 2$, a \textit{$d$-branch tree} is a rooted tree such that every branch vertex has at least $d$ children.

\begin{lemma}\label{lemma:essentialMonochromatic}
	For all integers $a_1, \dots, a_c \geq 0$ and $d \geq 2$, every $d$-branch tree with vertices coloured by $[c]$ and effective height at least $\sum_{i = 1}^c a_i$, contains an essentially $i$-monochromatic $d$-branch subtree of effective height at least $a_i$ for some $i \in [c]$. 
\end{lemma}
\begin{proof}
	We proceed by induction on $\sum_{i = 1}^c a_i$. The base case, $a_1 = \dots = a_c = 0$, is satisfied by taking a single vertex as the required $d$-branch subtree.
	
	Let $T$ be a $d$-branch tree of effective height $a_1 + \dots + a_c \geq 1$ with vertices coloured by $[c]$. Without loss of generality its effective root, $v$, has colour $1$. Let $v_1,\dots,v_d$ be children of $v$. Let $T_j$ be the subtree rooted at $v_j$. Note that $T_j$ has effective height at least $(a_1 - 1) + a_2 + \dots + a_c$. If, for some $j \in [d]$ and $i \in \{2, \dots, c\}$, $T_j$ contains an essentially $i$-monochromatic subtree of effective height $a_i$ then we are done. Otherwise, by induction, each $T_j$ contains an essentially $1$-monochromatic $d$-branch subtree of effective height $a_1 - 1$. These subtrees, together with $v$, are an essentially $1$-monochromatic $d$-branch subtree of $T$, as required.
\end{proof}

We now prove a lower bound on $\phi$ by using an essentially monochromatic subtree to find anagrams in sufficiently large trees.
\begin{theorem}\label{thm:effectiveDepthDegree}
	Let $T$ be a $d$-branch tree of effective height at least $h'$ and height at most $h \geq \max\{2, \sqrt{d}\}$. Then
	\begin{align*}
		\phi(T) \geq c:= \ceil{\sqrt{\frac{h'}{\log_dh}}\,\,}.
	\end{align*}
\end{theorem}

\begin{proof}
	If $c \leq 1$ the theorem follows trivially, so assume $c > 1$. Let $T$ be coloured with $x$ colours where $1 \leq x \leq c - 1$. Our goal is to show that $T$ contains an anagram. For $i \in [x]$ define $a_i \in \{\floor{h'/x}, \ceil{h'/x}\}$ such that $\sum_{i = 1}^x a_i = h'$. By Lemma \ref{lemma:essentialMonochromatic}, and without loss of generality, $T$ contains an essentially $1$-monochromatic $d$-branch subtree, $S$, of effective height at least $\floor{h'/x}$. 
	
	Let $r$ be the root of $S$. There are at least $d^{\floor{h'/x}}$ paths from $r$ to the leaves of $S$, and the colouring of each path defines a multiset of order at most $h + 1$. Since each path shares the colour of $r$, there are at most $h^{x}$ distinct multisets that can occur on the paths. Since $x \leq c - 1$,
	\begin{align*}
		\#\text{multisets} \leq h^{x} < h^{(c^2/x) - 2}
	\end{align*}
	Since $h \geq \sqrt{d}$
	\begin{align*}
		h^{(c^2/x) - 2}  \leq \frac{1}{d}h^{(c^2/x)}.
	\end{align*}
	Therefore
	\begin{align*}
		\#\text{multisets} < \frac{1}{d}h^{(c^2/x)} = \frac{1}{d}\lr{h^{\frac{1}{\log_d h}}}^{(h'/x)}  = d^{(h'/x) - 1} \leq d^{\floor{h'/x}} \leq \#\text{paths}.
	\end{align*}
	So there is a multiset that occurs on two different paths, $P_1$ and $P_2$, from $r$ to the leaves of $S$. Let $v$ be the lowest common vertex of $P_1$ and $P_2$, and let $\ell_i$ be the leaf endpoint of $P_i$. By definition, $M(P_1) = M(P_2)$ so $M(P_1 - P_2) = M(P_2 - P_1)$. Since $S$ is essentially $1$-monochromatic, the vertices $v$, $\ell_1$, and $\ell_2$ have colour $1$ so $((P_1 - P_2) \setminus \{\ell_1\})((P_2 - P_1) \setminus \{v\})$ is an anagram.
\end{proof}

\subsection{Bounds for subdivisions of the complete $d$-ary tree}

We now use Theorem \ref{thm:colorDaryTree} to prove an upper bound on $\phi$ for some subdivision of a given tree.
\begin{cor}\label{cor:leq_dary_chrom}
	For every $k \geq 0$, every complete $d$-ary tree of height $h'$, $T$, there exists a $(\leq k)$-subdivision, $S$, such that 
	\begin{align*}
			\phi(S) \leq c := 10\ceil{\frac{h'}{\log_{d + 1}(k/2d)}}.
	\end{align*}
\end{cor}
\begin{proof}
	Let $x := c/10$ and let $B \subseteq E(T)$ be the set of edges with depths $i\ceil{\frac{h'}{x}} - 1$ for $i \in \{0, \dots, x - 1\}$, recalling that the depth of an edge is the minimum depth of its endpoints. Let $F := T - B$ and note that $F$ is a forest where each component is a complete $d$-ary tree of height at most $\ceil{h'/x}$. Let $\mathcal{C}$ be the set of components of $F$. Root each component, $C \in \mathcal{C}$, at the vertex $r \in V(C)$ with minimum depth in $T$. The depth of $r$ is $i\ceil{h'/x}$ for some $i \in \{0, \dots, x - 1\}$. Define the \textit{depth} of $C$ to be $i$.
	
	By the definition of $c$ and $x$,
	\begin{align*}
		\log_{d + 1}\lr{\frac{k}{2d}} \geq \frac{h'}{x}.
	\end{align*}
	This implies
	\begin{align*}
		k \geq 2d(d+1)^{\frac{h'}{x}} \geq 2d(d+1)^{\ceil{\frac{h'}{x}} - 1}.
	\end{align*}
	Therefore, by Theorem \ref{thm:colorDaryTree}, for every $C \in \mathcal{C}$, there exist a $(\leq k)$-subdivision, $S_C$, with $\phi(S_C) \leq 10$ since $C$ has height at most $\ceil{h'/x}$. Anagram-free colour $S_C$ using colours $\{10i + 1, \dots, 10(i + 1)\}$ where $i$ is the depth of $C$. Let $S = B + \cup_{C \in \mathcal{C}} S_C$. Note that $S$ is a $(\leq k)$-subdivision of $T$ with a $10x$ colouring. We now show that this colouring of $S$ is anagram-free.
	
	Let $P$ be a subpath of $S$. Let $j \in \{0, \dots, x - 1\}$ be the minimum depth of component $C \in \mathcal{C}$ such that $S_C$ has non-empty intersection with $P$. By the construction of $S$, $P$ intersects with exactly one $C' \in \mathcal{C}'$ of depth $j$. Therefore, $P$ restricted to the colours of $C'$ corresponds to a subpath of $C'$ and, since $C'$ is angram-free, the restriction is not an anagram. Therefore, by Lemma \ref{lem:anagramFreeSubColor}, $P$ is not an anagram.
\end{proof}

The following lemma generalizes results for $(\leq k)$-subdivisions to $k$-subdivisions. Note that the $k$-subdivision a graph, $G$, is a subdivision of every $(\leq k)$-subdivision of $G$.
\begin{lemma}\label{lem:sub_chrom_4}
	If $S$ is a subdivision of $G$ then $\phi(S) \leq \phi(G) + 4$.
\end{lemma}
\begin{proof}
	Fix an anagram-free $\phi(G)$-colouring of $G$ and apply the colouring to the original vertices of $S$. The graph induced by the division vertices of $S$ is a forest of paths. Colour all of these paths with an anagram-free colouring on four new colours. By Lemma \ref{lem:anagramFreeSubColor}, this colouring of $S$ is anagram free.
\end{proof}

We now prove Theorem \ref{thm:dary_sub_chrom}, introduced in Section \ref{sec:tree}.
\begingroup
\def\thetheorem{\ref{thm:dary_sub_chrom}}
\begin{theorem}
	The $k$-subdivision, $S$, of the complete $d$-ary tree of height $h'$ satisfies
	\begin{align*}
		\sqrt{\frac{h'}{\log_{\min\{d, (h'(k + 1))^2\}}(h'(k + 1))}} \leq \phi(S) \leq \frac{10h'}{\log_{d + 1}\lr{k/2d}} + 14.
	\end{align*}
\end{theorem}
\addtocounter{theorem}{-1}
\endgroup
\begin{proof}
	Theorem \ref{thm:effectiveDepthDegree} proves the lower bound. Corollary \ref{cor:leq_dary_chrom} with Lemma \ref{lem:sub_chrom_4} prove the upper bound.
\end{proof}

\section{Subdivisions of general graphs}

Now we construct subdivisions of arbitrary graphs with bounded anagram-free chromatic number. Let $t = t_1, t_2, \dots$ be a sequence of positive integers. A subdivision, $S$, of a graph $G$, is a \textit{$t$-sequence-subdivision of $G$} if there is a bijection, $\ell: V(G) \to [|V(G)|]$, that satisfies the following two conditions. The first condition is that there is a proper $2$-colouring of $G$, with colours white and black, such that $\ell(u) > \ell(v)$ for every white vertex $u \in V(G)$ and black vertex $v \in V(G)$. The second condition requires some definitions. For every edge, $e \in E(G)$, define $w(e)$ to be the white vertex incident with $e$ and $b(e)$ be the black vertex incident with $e$. Define the bijection, $\ell': E(G) \to [|E(G)|]$, that orders edges in $E(G)$, first by the label of their white endpoint and second by the label of their black endpoint. That is, $\ell'(x) > \ell'(y)$ for edges $x, y \in E(G)$ if $\ell(w(x)) > \ell(w(y))$ or if $\ell(w(x)) = \ell(w(y))$ and $\ell(b(x)) > \ell(b(y))$. Note that $\ell'$ is determined by $\ell$. Now, the second condition on $\ell$ is that every edge, $e \in E(G)$, has $3t_{\ell'(e)}$ division vertices.

Let $G$ be a graph, $t$ be a sequnce of positive integers, and $S$ be a $t$-sequence-subdivision $G$ with corresponding vertex and edge labellings $\ell$ and $\ell'$. Define functions $X$, $Y$, and $Z$ such that for every edge, $e \in E(G)$, $X(e)$, $Y(e)$, and $Z(e)$ are pairwise disjoint subpaths of the path replacing $e$ with $|V(X(e))|=|V(Y(e))|=|V(Z(e))|=t_{\ell'(e)}$, $X(e)$ is adjacent to the white end of $e$, and $Z(e)$ is adjacent to the black end of $e$. Define sets of these paths, $\mathcal{X} := X(E(G))$, $\mathcal{Y} := Y(E(G))$, and $\mathcal{Z} := Z(E(G))$. A vertex colouring of $S$ is \textit{discriminating} if the following conditions hold.
\begin{enumerate}[(1)]
	\item The original vertices of $S$ are coloured by the proper $2$-colouring of $G$, and these two colours only occur on the original vertices. \label{cond:blackWhite}
	\item Every anagram in $S$ contains at least one original vertex. \label{cond:containOrig}
	\item For all $Q \in \{X, Y, Z\}$ there exists a nonempty set of colours, $C(Q)$, that occur only on the vertices of paths in $Q(E(G))$. \label{cond:uniqueColour}
	\item For all $Q \in \{X, Y, Z\}$ and $q \in E(G)$, \label{cond:big}
	\begin{align*}
		\sum_{e \in E(G): \ell'(e) < \ell'(q)} |V_{C(Q)}(Q(e))| < |V_{C(Q)}(Q(q))|.
	\end{align*}
\end{enumerate}
Note that whether $S$ has a discriminating vertex colouring depends on the sequence $t$. For example, the sequence $t_i = 1$, for all $i$, causes condition $4$ to fail for sufficiently large $G$.

\begin{theorem}\label{thm:graphSub_general}
	Let $S$ be a $t$-sequence-subdivision of a graph $G$ with sequence $t$. Every discriminating vertex colouring of $S$ is anagram-free.
\end{theorem}
\begin{proof}
	Let $\ell$ and $\ell'$ be the associated vertex and edge labellings of $G$. Let $f$ be a discriminating vertex colouring of $S$.
	
	Let $P$ be a path in $S$ and assume for the sake of contradiction that $P$ is an anagram. By Condition (\ref{cond:containOrig}), $V(P)$ contains at least one original vertex. Since $G$ is properly $2$-coloured, all subpaths of $G$ that are anagrams have order at least $4$. The $2$-colouring of $G$ is applied to the original vertices of $S$, so, by Lemma \ref{lem:anagramFreeSubColor}, $P$ contains at least four original vertices. Therefore $P$ has at least one subpath from each of $\mathcal{X}$, $\mathcal{Y}$, and $\mathcal{Z}$. Let $x, y, z \in E(G)$ be the edges maximizing $\ell'$ such that $V(P) \cap V(X(x)) \neq \emptyset$, $V(P) \cap V(Y(y)) \neq \emptyset$, and $V(P) \cap V(Z(z)) \neq \emptyset$. 
	
	A path, $P'$, \textit{partially intersects} $P$ if $V(P') \nsubseteq V(P)$ and $V(P') \cap V(P) \neq \emptyset$. There are most two paths in $\mathcal{X}$, $\mathcal{Y}$, and $\mathcal{Z}$ that partially intersect $P$ since every division vertex has degree $2$ in $S$. Therefore at least one of $X(x)$, $Y(y)$, and $Z(z)$ is a subpath of $P$. Define $q \in \{x, y, z\}$ and $Q \in \{X, Y, Z\}$ such that $Q(q) \in \{X(x), Y(y), Z(z)\}$ is a subpath of $P$. Since $f$ is a discriminating colouring
	\begin{align*}
		\sum_{e \in E(G): \ell'(e) < \ell'(q)} |V_{C(Q)}(Q(e))| < |V_{C(Q)}(Q(q))|.
	\end{align*}
	Therefore, by the maximality of $\ell'(q)$, there are more vertices in $Q(q)$ coloured by $C(Q)$ than there are vertices coloured by $C(Q)$ in the rest of of $P$. Thus $|V_{C(Q)}(Q(q))| > \frac{1}{2}|V_{C(Q)}(P)|$. Let $LR$ be the split of $P$. By Lemma \ref{lem:anagramFreeSubColor}, $V_{C(Q)}(L) = V_{C(Q)}(R) = \frac{1}{2}|V_{C(Q)}(P)|$ so both $L$ and $R$ intersect $Q(q)$. Therefore the midedge of $P$ is an edge of $Q(q)$. Since the midedge of $P$ is unique, exactly one of $X(x)$, $Y(y)$, and $Z(z)$ is a subpath of $P$.
	
	Since $G$ is properly $2$-coloured, every subpath of $G$ that is an anagram has a white endpoint and a black endpoint. Therefore one of the endmost original vertices of $P$ is white, call this vertex $\alpha$. Since $P$ partially intersects exactly two of $X(x)$, $Y(y)$, and $Z(z)$, there is a black vertex $\beta \in N_\alpha(G)$ such that $\alpha\beta \in \{x, y, z\}$, where $N_\alpha(G)$ is the neighbourhood of $\alpha$. Recall that both $L$ and $R$ contain at least two original vertices and the midpoint of $P$ is in $Q(q)$. Therefore neither endpoint of $q$ is an endmost original vertex of $P$, so $\alpha \neq w(q)$. Also, there is a black vertex, $\gamma \in N_\alpha(G)$, such that the division vertices of $\alpha\gamma$ are all in $P$. Since $\alpha\beta \in \{x, y, z\}$ and $\alpha\beta \neq q$, there is an $A \in \{X, Y, Z\}$ such $A(\alpha\beta) \in \{X(x), Y(y), Z(z)\}$, for some $A \neq Q$. Now, $\ell'(\alpha\beta) > \ell'(q)$ because $A(q)$ is a subpath of $P$ and $\ell'(\alpha\beta)$ is maximal. Therefore $\ell(\alpha) > \ell(w(q))$, so $\ell'(\alpha\gamma) > \ell'(q)$. This contradicts the maximality of $\ell'(q)$ because $Q(\alpha\gamma)$ is a subpath of $P$.
\end{proof}
We now use Theorem \ref{thm:graphSub_general} to prove Theorem \ref{thm:graphSub_constant_nice}.
\begingroup
\def\thetheorem{\ref{thm:graphSub_constant_nice}}
\begin{theorem}
	Every graph $G'$ has a $(\leq 3(2)^{2|E(G')| - 1})$-subdivision, $S$, with $\phi(S) \leq 14$.
\end{theorem}
\addtocounter{theorem}{-1}
\endgroup

\begin{proof}
	Let $G$ be the $1$-subdivision of $G'$, note that $G$ has a proper $2$-colouring. Define the sequence $t$ by $t_i = 2^{i-1}$ for $i \geq 1$. Let $S$ be a $t$-sequence-subdivision of $G$. Since $G$ has $2|E(G')|$ edges and $3t_{2|E(G')|} = 3(2)^{2|E(G')| - 1}$, $S$ satisfies the bound on division vertices per edge required by the theorem. Let $\ell$ and $\ell'$ be the associated vertex and edge labellings of $G$.
	
	Let $f$ be the vertex colouring of $S$ defined as follows. Colour the original vertices of $S$ with the proper $2$-colouring of $G$ that corresponds to $\ell$. Assign a disjoint set of four colours to each of $\mathcal{X}$, $\mathcal{Y}$ and $\mathcal{Z}$. Colour each of the paths in $\mathcal{X}$, $\mathcal{Y}$ and $\mathcal{Z}$ with an anagram-free $4$-colouring with their assigned set of four colours.
	
	We now show that $f$ is discriminating. Conditions (\ref{cond:blackWhite}) and (\ref{cond:uniqueColour}) are satisfied trivially. Condition (\ref{cond:containOrig}) is satisfied because each of the paths in $\mathcal{X}$, $\mathcal{Y}$ and $\mathcal{Z}$ is anagram-free and they use their own set of colours so every anagram in $S$ contains an original vertex. Condition (\ref{cond:big}) is satisfied because for all $Q \in \{X, Y, Z\}$ and $q \in E(G)$, $|V_{C(Q)}(Q(q))| = |V(Q(q))|$, and
	\begin{align*}
		\sum_{e \in E(G) : \ell'(e) < \ell'(q)} |V(Q(e))| = 2^{\ell'(q) - 2} + \dots + 1 =  2^{\ell'(q) - 1} - 1 < 2^{\ell'(q) - 1} = |V(Q(q))|.
	\end{align*}
	Therefore $f$ is an anagram-free $14$-colouring of $S$.
\end{proof}

We use Theorem \ref{thm:graphSub_constant_nice} to bound $\phi$ on subdivisions of graphs in terms of division vertices per edge.
\begin{theorem}\label{thm:graphSub_chromatic}
	For every graph $G$ and $k \in \mathbb{Z}^+$ there exists a $(\leq 3(4)^{\lceil|E(G)|/k\rceil})$-subdivision, $S$, of $G$ with $\phi(S) \leq 2 + 12k$.
\end{theorem}
\begin{proof}
	Take $k$ subgraphs of $G$ with an equitable number of edges per subgraph. Subdivide them and colour them using Theorem \ref{thm:graphSub_constant_nice}. Merge these subdivisions to obtain an anagram-free $2 + 12k$ colouring of $G$.
\end{proof}

We now optimize our use of Theorem \ref{thm:graphSub_general} to improve the upper bound on $\phi$.
\begingroup
\def\thetheorem{\ref{thm:graphSub_constant_optimized}}
\begin{theorem}
	Every graph $G'$ has a $\lr{\leq 45\lr{1 + \frac{75}{9}}^{2|E(G')| - 1}}$-subdivision, $S$, with $\phi(S) \leq 8$.
\end{theorem}
\addtocounter{theorem}{-1}
\endgroup

\begin{proof}
	Let $G$ be the $1$-subdivision of $G'$, note that $G$ has a proper $2$-colouring. Define the sequence $t$ with $t_1 = 8$ and
	\begin{align}\label{eqn:graphSub_sequence}
	t_n = 15 + \floor{\frac{25}{3}\sum_{i=1}^{n-1} t_i}.
	\end{align}
	Let $S$ be a $t$-sequence-subdivision of $G$. It is straightforward to verify that $t_n \leq 15\lr{1 + \frac{75}{9}}^{n - 1}$ so $S$ satisfies the limit on division vertices per edge required by the theorem. Let $\ell$ and $\ell'$ be the associated vertex and edge labellings of $G$.
	
	Define the colouring $f:V(S) \to \{1,2,3, 4, 5, 6, \text{white}, \text{black}\}$ as follows. Original vertices are coloured white or black according to $\ell$. For every $e \in E(G)$ define $P_e = v_1\dots v_{3t_{\ell'(e)}}$ to be the division vertices of $e$. Let $W$ be an anagram-free word on $\{1,2,3,4\}$ of length $3\ell'(e)$ and colour $P_e$ as follows. For all $v_i \in V(P_e)$, if $W_i \in \{1,2,3\}$ then $f(v_i) := W_i$. Otherwise, $f(v_i) := 4$ if $v_i \in V(X(e))$, $f(v_i) := 5$ if $v_i \in V(Y(e))$, and $f(v_i) := 6$ if $v_i \in V(Z(e))$. 
	
	We now show that $f$ is discriminating. Condition (\ref{cond:blackWhite}) is satisfied trivially. Condition (\ref{cond:containOrig}) is satisfied because $P_e$ is coloured by an anagram-free word for all $e \in E(G)$. Condition (\ref{cond:uniqueColour}) is satisfied by $C(X) = \{4\}$, $C(Y) = \{5\}$, and $C(Z) = \{6\}$. We now show that Condition (\ref{cond:big}) is satisfied.
	
	Let $Q \in \{X, Y, Z\}$ and $q \in E(G)$. The same symbol cannot occur twice in a row so $|V_{C(Q)}(Q(q))| \leq \frac{5}{9}|V(Q(q))|$, since $|V(Q(q))| \geq 8$. Therefore
	\begin{align*}
		\sum_{e \in E(G) : \ell'(e) < \ell'(q)} |V_{C(Q)}(Q(e))| \leq \frac{5}{9}\sum_{e \in E(G) : \ell'(e) < \ell'(q)} |V(Q(e))|.
	\end{align*}
	Every anagram-free word of length $8$ contains at least four distinct symbols. Therefore $|V_{C(Q)}(Q(q))| \geq \frac{1}{15}|V(Q(q))|$. By (\ref{eqn:graphSub_sequence})
	\begin{align*}
		 \frac{5}{9}\sum_{e \in E(G) : \ell'(e) < \ell'(q)} |V(Q(e))| =  \frac{5}{9}\sum_{i=1}^{n-1} t_i \leq \frac{1}{15}t_n - 1.
	\end{align*}
	Therefore
	\begin{align*}
		 \sum_{e \in E(G) : \ell'(e) < \ell'(q)} |V_{C(Q)}(Q(e))| \leq \frac{5}{9}\sum_{i=1}^{n-1} t_i < \frac{1}{15}t_n = \frac{1}{15}|V(Q(e))| \leq |V_{C(Q)}(Q(q))|.
	\end{align*}
	Thus condition $4$ is satisfied so $f$ is an anagram-free $8$-colouring of $S$.
\end{proof}
Theorem \ref{thm:graphSub_constant_optimized} uses simple bounds on the density of symbols in anagram-free words. Better bounds on density would improve the base of the bound in Theorem \ref{thm:graphSub_constant_optimized}.

\subsection{Subdivisions of complete graphs}
Recall that $\mathcal{M}_{k,c}$ is the set of colour multisets on $c$ symbols of size $k$ and that $\mathcal{M}_{\leq k,c}$ is the set of colour multisets of $c$ symbols of size at most $k$.

\begingroup
\def\thetheorem{\ref{thm:K_n_anagram_lowerBound}}
\begin{theorem}
	Let $S$ be a $\lr{\leq k}$-subdivision of $K_n$. If $S$ is anagram-free $c$-colourable then
	\begin{align*}
		k \geq \lr{c!\lr{\frac{n}{c} - 1}}^{1/c} - c.
	\end{align*}
\end{theorem}
\addtocounter{theorem}{-1}
\endgroup
\begin{proof}
	Suppose for the sake of contradiction that 
	\begin{align}\label{eqn:dsubdibision}
		k < \lr{c!\lr{\frac{n}{c} - 1}}^{1/c} - c.
	\end{align}
	Fix an anagram-free colouring of $S$. Colour each edge $e \in E(K_n)$ with the colour multiset of the subdivision vertices of $e$ in $S$ and colour each vertex of $K_n$ with its colour in $S$. Note that there are
	\begin{align*}
		|\mathcal{M}_{\leq k,c}| = \sum^{k}_{i = 0} \binom{i + c - 1}{c - 1} = \binom{k + c}{c} \leq \frac{(k+c)^c}{c!}
	\end{align*}
	possibilities for the colour of each edge. Let $x := \ceil{n/c}$ and $G$ be a vertex-monochromatic $K_x$ subgraph of $K_n$. Note that
	\begin{align*}
		|E(G)| = \frac{x}{2}(x - 1) \geq \frac{x}{2}\lr{\frac{n}{c} - 1}.
	\end{align*}
	Therefore, by (\ref{eqn:dsubdibision})
	\begin{align*}
		|E(G)| \geq \frac{x}{2}\lr{\frac{n}{c} - 1} > \frac{x}{2}\frac{(k+c)^c}{c!} \geq \frac{x}{2}|\mathcal{M}_{\leq k, c}| \geq \frac{x}{2}\#\text{colours}.
	\end{align*}
	So there is a set of more than $x/2$ edges that have the same colour. Therefore there is a vertex, $v \in V(G)$, that is incident with at least two edges, $\alpha, \beta \in E(G)$, with the same colour. Let $u$ be the other endpoint of $\alpha$, $P_\alpha$ be the path induced by the division vertices of $\alpha$, and $P_\beta$ be the path induced by the division vertices of $\beta$. Then $u P_\alpha v P_\beta$ is an anagram in $S$.
\end{proof}

\def\soft#1{\leavevmode\setbox0=\hbox{h}\dimen7=\ht0\advance \dimen7
  by-1ex\relax\if t#1\relax\rlap{\raise.6\dimen7
  \hbox{\kern.3ex\char'47}}#1\relax\else\if T#1\relax
  \rlap{\raise.5\dimen7\hbox{\kern1.3ex\char'47}}#1\relax \else\if
  d#1\relax\rlap{\raise.5\dimen7\hbox{\kern.9ex \char'47}}#1\relax\else\if
  D#1\relax\rlap{\raise.5\dimen7 \hbox{\kern1.4ex\char'47}}#1\relax\else\if
  l#1\relax \rlap{\raise.5\dimen7\hbox{\kern.4ex\char'47}}#1\relax \else\if
  L#1\relax\rlap{\raise.5\dimen7\hbox{\kern.7ex
  \char'47}}#1\relax\else\message{accent \string\soft \space #1 not
  defined!}#1\relax\fi\fi\fi\fi\fi\fi}


\begin{thebibliography}{17}
\providecommand{\natexlab}[1]{#1}
\providecommand{\url}[1]{\texttt{#1}}
\expandafter\ifx\csname urlstyle\endcsname\relax
  \providecommand{\doi}[1]{doi: #1}\else
  \providecommand{\doi}{doi: \begingroup \urlstyle{rm}\Url}\fi

\bibitem[Alon et~al.(2002)Alon, Grytczuk, Ha{\l}uszczak, and
  Riordan]{alon2002nonrepetitive}
Noga Alon, Jaros{\l}aw Grytczuk, Mariusz Ha{\l}uszczak, and Oliver Riordan.
\newblock Nonrepetitive colorings of graphs.
\newblock \emph{Random Structures \& Algorithms}, 21\penalty0 (3-4):\penalty0
  336--346, 2002.

\bibitem[Bar{\'a}t and Wood(2008)]{barat2008notes}
J{\'a}nos Bar{\'a}t and David~R Wood.
\newblock Notes on nonrepetitive graph colouring.
\newblock \emph{Electron. J. Combin}, 15\penalty0 (1):\penalty0 R99, 2008.

\bibitem[Bre{\v{s}}ar et~al.(2007)Bre{\v{s}}ar, Grytczuk, Klav{\v{z}}ar,
  Niwczyk, and Peterin]{brevsar2007nonrepetitivetree}
Bo{\v{s}}tjan Bre{\v{s}}ar, Jaros{\l}aw Grytczuk, Sandi Klav{\v{z}}ar, Staszek
  Niwczyk, and Iztok Peterin.
\newblock Nonrepetitive colorings of trees.
\newblock \emph{Discrete Mathematics}, 307\penalty0 (2):\penalty0 163--172,
  2007.

\bibitem[Britten and Davidson(1971)]{britten1971repetitive}
Roy~J Britten and Eric~H Davidson.
\newblock Repetitive and non-repetitive {DNA} sequences and a speculation on
  the origins of evolutionary novelty.
\newblock \emph{Quarterly Review of Biology}, 46\penalty0 (2):\penalty0
  111--138, 1971.

\bibitem[Cummings(1996)]{cummings1996strongly}
Larry~J Cummings.
\newblock Strongly square-free strings on three letters.
\newblock \emph{Australasian J. Combinatorics}, 14:\penalty0 259--266, 1996.

\bibitem[Dujmovi{\'c} et~al.(2016)Dujmovi{\'c}, Joret, Kozik, and
  Wood]{dujmovic2011nonrepetitive}
Vida Dujmovi{\'c}, Gwena{\"e}l Joret, Jakub Kozik, and David~R Wood.
\newblock Nonrepetitive colouring via entropy compression.
\newblock \emph{Combinatorica}, 36\penalty0 (6):\penalty0 661--686, 2016.

\bibitem[Grytczuk()]{grytczuk2007nonrepetitive}
Jaros{\l}aw Grytczuk.
\newblock Nonrepetitive colorings of graphs a survey.
\newblock \emph{Int. J. Math. Math. Sci.}, 2007:\penalty0 Art. ID 74639.

\bibitem[Grytczuk(2007)]{grytczuk2006nonrepetitive}
Jaros{\l}aw Grytczuk.
\newblock Nonrepetitive graph coloring.
\newblock In \emph{Graph Theory in Paris}, Trends in Mathematics, pages
  209--218. Birkhauser, 2007.

\bibitem[Grytczuk et~al.(2013)Grytczuk, Kozik, and Micek]{grytczuk2013new}
Jaros{\l}aw Grytczuk, Jakub Kozik, and Piotr Micek.
\newblock New approach to nonrepetitive sequences.
\newblock \emph{Random Structures \& Algorithms}, 42\penalty0 (2):\penalty0
  214--225, 2013.

\bibitem[Harant and Jendro{\soft{l}}(2012)]{harant2012nonrepetitive}
Jochen Harant and Stanislav Jendro{\soft{l}}.
\newblock Nonrepetitive vertex colorings of graphs.
\newblock \emph{Discrete Math.}, 312\penalty0 (2):\penalty0 374--380, 2012.
\newblock \doi{10.1016/j.disc.2011.09.027}.

\bibitem[Kam{\v{c}}ev et~al.(2017)Kam{\v{c}}ev, {\L}uczak, and
  Sudakov]{kamvcev2016anagram}
Nina Kam{\v{c}}ev, Tomasz {\L}uczak, and Benny Sudakov.
\newblock Anagram-free colorings of graphs.
\newblock \emph{Combinatorics, Probability and Computing}, 2017.
\newblock \doi{10.1017/S096354831700027X}.

\bibitem[Ker{\"a}nen(1992)]{keranen1992abelian}
Veikko Ker{\"a}nen.
\newblock Abelian squares are avoidable on $4$ letters.
\newblock In \emph{Automata, languages and programming}, volume 623 of
  \emph{Lecture Notes in Comput. Sci.}, pages 41--52. Springer, 1992.
\newblock \doi{10.1007/3-540-55719-9\_62}.

\bibitem[Ker{\"a}nen(2009)]{keranen2009abelian}
Veikko Ker{\"a}nen.
\newblock A powerful abelian square-free substitution over 4 letters.
\newblock \emph{Theoret. Comput. Sci.}, 410\penalty0 (38-40):\penalty0
  3893--3900, 2009.
\newblock \doi{10.1016/j.tcs.2009.05.027}.

\bibitem[Ne{\v{s}}et{\v{r}}il et~al.(2012)Ne{\v{s}}et{\v{r}}il, de~Mendez, and
  Wood]{nevsetvril2012characterisations}
Jaroslav Ne{\v{s}}et{\v{r}}il, Patrice~Ossona de~Mendez, and David~R Wood.
\newblock Characterisations and examples of graph classes with bounded
  expansion.
\newblock \emph{European J. Combinatorics}, 33\penalty0 (3):\penalty0 350--373,
  2012.

\bibitem[Pezarski and Zmarz(2009)]{pezarski2009non}
Andrzej Pezarski and Micha{\l} Zmarz.
\newblock Non-repetitive 3-coloring of subdivided graphs.
\newblock \emph{Electron. J. Combin}, 16\penalty0 (1):\penalty0 N15, 2009.

\bibitem[Thue(1914)]{thue1914probleme}
Axel Thue.
\newblock Probleme {\"u}ber ver{\"a}nderungen von zeichenreihen nach gegebenen
  regeln.
\newblock pages I. Math. naturv. Klasse, 10. Christiana Videnskabs-Selskabs
  Skrifte, 1914.

\bibitem[Wilson and Wood(2016)]{wilson2016abelian}
Tim~E Wilson and David~R Wood.
\newblock Abelian square-free graph colouring.
\newblock \emph{arXiv preprint arXiv:1607.01117}, 2016.

\end{thebibliography}
\end{document}